\documentclass[a4paper,12 pt]{amsart}

\usepackage{amsfonts}
\usepackage{amsmath}
\usepackage{amsmath,amsfonts,amssymb,amsthm,latexsym}
\usepackage{graphicx}
\usepackage{hyperref}
 
\newcommand{\RR}{\mathbb{R}}
\newcommand{\CC}{\mathbb{C}}
\newcommand{\NN}{\mathbb{N}}
\newcommand{\ZZ}{\mathbb{Z}}

\DeclareMathOperator{\Span}{span}

\newtheorem{Tw}{Theorem}
\newtheorem{Le}{Lemma}
\newtheorem{St}{Proposition} 

\newtheorem{Wn}{Corollary}

\theoremstyle{remark}
\newtheorem{Uw}{Remark}

\theoremstyle{definition}
\newtheorem{Df}{Definition}

\begin{document}
\title[Spherical polyharmonics and Poisson kernels]{Spherical polyharmonics and Poisson kernels for polyharmonic functions}
\author{Hubert Grzebu{\l}a}
\address{Faculty of Mathematics and Natural Sciences,
College of Science\\
Cardinal Stefan Wyszy\'nski University\\
W\'oycickiego 1/3,
01-938 Warszawa, Poland}
\email{h.grzebula@student.uksw.edu.pl}
\author{S{\l}awomir Michalik}
\address{Faculty of Mathematics and Natural Sciences,
College of Science\\
Cardinal Stefan Wyszy\'nski University\\
W\'oycickiego 1/3,
01-938 Warszawa, Poland}
\email{s.michalik@uksw.edu.pl}
\urladdr{\url{http://www.impan.pl/~slawek}}
 
\begin{abstract}
We introduce and develop the notion of spherical polyharmonics, which are a natural generalisation of spherical harmonics.
In particular we study the theory of zonal polyharmonics, which allows us, analogously to zonal harmonics, to construct Poisson kernels
for polyharmonic functions on the union of rotated balls.
We find the representation of Poisson kernels and zonal polyharmonics in terms of the Gegenbauer polynomials.
We show the connection between the classical Poisson
kernel for harmonic functions on the ball, Poisson kernels for polyharmonic functions on the union of rotated balls, and the Cauchy-Hua kernel
for holomorphic functions on the Lie ball.
 \end{abstract}

\maketitle
\section{Introduction}
In the paper we introduce the notion of so called 
spherical polyharmonics of degree $m$ and of order $p$ as $m$-homogeneous and
$p$-polyharmonic polynomials restricted to the
 union of the rotated euclidean unit spheres
$\widehat{S}_p:=\bigcup_{k=0}^{p-1}e^{\frac{k\pi i}{p}}S$, where $m,p\in\NN$. We develop the theory of spherical polyharmonics in a similar way to that of
spherical harmonics given in \cite[Section 5]{A-B-R}. The crucial role in our considerations plays the Hilbert space $L^2(\widehat{S}_p)$  of
square-integrable functions on $\widehat{S}_p$
with an inner product defined by
\begin{equation}
\label{I:3}
\left\langle f,g\right\rangle _{ \widehat{S}_p }:=\frac{1}{p} \int\limits_S \sum_{j=0}^{p-1}f(e^{\frac{j\pi i}{p}} \zeta )
\overline{g(e^{\frac{j\pi i}{p}} \zeta )}\,d\sigma (\zeta),
\end{equation}
where $\sigma$ is the normalised surface-area measure on $S$ (so that $\sigma(S)=1$).

The introduction of such space allows us to prove the polyharmonic version of the theorem about spherical harmonic decomposition of
$L^2(S)$ (see \cite[Theorem 5.12]{A-B-R}). Namely we prove that $ L^2 ( \widehat{S}_p )  $ is a direct sum of the spaces
$\mathcal{H}_m^p ( \widehat{S}_p)$ of spherical polyharmonics  of degree $m$ and of order $p$ (Theorem \ref{Tw:1}):
\begin{equation*}
L^2 ( \widehat{S}_p )=\bigoplus_{m=0}^{\infty} \mathcal{H}_m^p ( \widehat{S}_p ).
\end{equation*}

Due to this theorem we can consider the space $\mathcal{H}_m^p(\widehat{S}_p)$ as a finitely dimensional Hilbert space with
the inner product (\ref{I:3})
induced from $L^2 ( \widehat{S}_p )$. It allows us to introduce zonal polyharmonics $Z^p_m(x,\zeta)$ in a similar way to zonal harmonics.

After developing the theory of zonal polyharmonics, we obtain the expected formulas for Poisson kernels for the union of the rotated euclidean
unit balls $\widehat{B}_p:=\bigcup_{k=0}^{p-1}e^{\frac{k\pi i}{p}}B$ (Theorem \ref{Wn:7}):
\begin{equation*}
P_p(x,\zeta)=\sum_{m=0}^{\infty}Z^p_m(x,\zeta)=\frac{1-|x|^{2p}}{(x^2\overline{\zeta}^2-2x\cdot \overline{\zeta}+1)^{n/2}}\quad \textrm{for}
\quad x\in \widehat{B}_p,\ \zeta \in \widehat{S}_p.
\end{equation*} 

In the rest of the paper we study the properties of these kernels.
In particular, using the Gegenbauer polynomials, we will find series expansions of Poisson kernels $P_p(x,\zeta)$ and
zonal polyharmonics $Z_m^p(z,\zeta)$. We will also show that Poisson kernels $P_p(x,\zeta)$ for $\widehat{B}_p$ are, in some sense,
intermediate kernels between the classical Poisson kernel for the unit ball $B$ in $\RR^n$ and the Cauchy-Hua kernel for the Lie ball in $\CC^n$. 

The motivation for the study of spherical polyharmonics comes from the Dirichlet-type problem for polyharmonic complex functions,
which has been examined in our previous work \cite{G-M}. That problem, inspired by papers \cite{L} and \cite{S-M}, is different from the classical
boundary value problems for polyharmonic functions (see below).

Namely, the problem is following:
find a polyharmonic function $u$ of order $p$ on
$\widehat{B}_p$, such that $u$ is continuous in $\widehat{B}_p\cup \widehat{S}_p$ 
and satisfies the boundary conditions
$u(x)=f(x)$ for $x\in \widehat{S}_p$, 
where $ p\in \NN $ and the function $f$ is given and continuous in $\widehat{S}_p$.
This problem is concisely written as
\begin{equation}
\label{I:1}
   \begin{cases}
   \Delta^pu(x)=0, & \text{} x\in \widehat{B}_p \\
   u(x)=f(x), & \mbox{} x\in \widehat{S}_p.
   \end{cases}
   \end{equation}
\cite[Theorem 1]{G-M} states that the solution of (\ref{I:1}) is unique and is given by the sum of Poisson type integrals:
\begin{equation}
\label{I:2}
u(x) =\frac{1}{p}\sum_{k=0}^{p-1} 
\int\limits_{S} \frac{1-|x|^{2p}}{|e^\frac{-k\pi i}{p}x-\zeta|^n }f(e^\frac{k\pi i}{p}\zeta) \,d\sigma(\zeta).
\end{equation}

In particular for $p=1$ the problem (\ref{I:1}) reduces to the classical Dirichlet problem for harmonic functions on the ball $B$,
which solution has a form 
\begin{equation*}
u(x) =
\int\limits_{S} \frac{1-|x|^{2}}{|x-\zeta|^n }f(\zeta) \,d\sigma(\zeta),
\end{equation*} 
where the function $P(x,\zeta) = \frac{1-|x|^{2}}{|x-\zeta|^n }$
is the classical Poisson kernel for the unit ball $B$. Here the question arises how to define a Poisson type kernel for the union of
rotated balls $\widehat{B}_p$,
because the  formula (\ref{I:2}) does not supply information about it.

However, we know also from the theory of harmonic functions that the Poisson kernel for the ball can be expressed in terms of zonal harmonics,
which are a particular case of spherical harmonics. It prompts us to introduce the space of spherical polyharmonics
$\mathcal{H}_m^p(\widehat{S}_p)$ and to study their properties.

It is worth to mention that boundary value problems for polyharmonic functions have recently been extensively investigate,
see for example \cite{B-D-W}, \cite{D-Q-W} or \cite{G-G-S} and the references given there. Unlike the problem (\ref{I:1}),
in the mentioned papers the boundary conditions consist of differential operators (for instance the normal derivatives or iterated Laplacians).
In particular in the paper \cite{D-Q-W} the authors consider the problem with $L^m$ boundary data in the euclidean unit ball:
\begin{equation}
\label{Du}
   \begin{cases}
   \Delta^pu(x)=0, & \text{} x\in B, \\
   \Delta^ju(x)=f_j(x), & \mbox{} x\in S.
   \end{cases}
   \end{equation} 
where $0\leq j<p$ and $1\leq m\leq \infty $. They introduce the sequence of higher order Poisson kernels which satisfy appropriate conditions
(see Definition 2.1 in \cite{D-Q-W}) and they give the integral representation solutions of the problem (\ref{Du}).
The kernels and the integral representation solutions given there are different from our ones because they are closely related to the problem
(\ref{Du}), whereas our problem (\ref{I:1}) is completely different. 

The paper is organised as follows. The next section consists of some basic notations which we use throughout this paper.

In Section 3 we examine the polyharmonic polynomials. We recall  some important properties of harmonic polynomials and we prove their analogues
for polyharmonic polynomials. We also calculate the dimension of the space of polyharmonic polynomials (Proposition \ref{Wn:3}).

In Section 4 we introduce the notion of spherical polyharmonics. We prove their properties
(Corollaries \ref{Wn:1}--\ref{Wn:2} and Propositions \ref{Le:4}--\ref{Le:6}). We also prove the main theorem 
about spherical polyharmonic decomposition of $ L^2 ( \widehat{S}_p )$ (Theorem \ref{Tw:1}).

In Section 5 we introduce zonal polyharmonics (Definition \ref{Df:5}). We find the connection between zonal polyharmonics
and zonal harmonics (Theorem \ref{Tw:2}), and we find the orthogonal decomposition functions from $L^2(\widehat{S}_p )$
(Theorem \ref{Tw:3}).

In Section 6 we define Poisson kernels for the union of rotated balls and Poisson integrals for functions continuous on $\widehat{S}_p$.
We find a counterpart of  \cite[Proposition 5.31]{A-B-R} for zonal polyharmonics (Proposition \ref{Tw:4}).
Finally we find formulas for the Poisson kernels for the sum of rotated balls (Theorem \ref{Wn:7}) and we give some properties of these kernels
(Proposition \ref{Tw:5}). As an application we solve the Dirichlet-type problem given by (\ref{I:1}) (Theorem \ref{Wn:9}).

In the next section we find an explicite formula for zonal polyharmonics (Theorem \ref{Tw:7}). We use here the Gegenbauer polynomials
and a generating formula for them.
 
In the last section we find the connection between the Poisson kernels and the Cauchy-Hua kernel. As a corollary, we conclude that for every
holomorphic function on the Lie ball and continuous on the Lie sphere there exist
polyharmonic functions, which are convergent to it (Theorem \ref{th:4}).

\section{Preliminaries}
In this section we give some basic notations and definitions.

We define the real norm
\begin{equation*}
|x|= ( \sum_{j=1}^nx_j^2 )^{1/2}\quad\textrm{for}\quad x=(x_1,\dots,x_n)\in \RR^n
\end{equation*}
and the complex norm
\begin{equation*}||z||= (\sum_{j=1}^n|z_j|_{\CC}^2 )^{1/2} \quad\textrm{for}\quad z=(z_1,\dots,z_n)\in \CC^n
\end{equation*} 
with $ |z_j|_{\CC}^2=z_j \overline{z}_j$. We will also use the complex extension of the real norm for complex vectors: 
\begin{gather*}
|z|=( \sum_{j=1}^nz_j^2 )^{1/2}\quad\textrm{for}\quad z=(z_1,\dots,z_n)\in\CC^n.
\end{gather*}
By a square root in the above formula we mean the principal square root, where a branch cut is taken along the non-positive real axis.
Obviously the function $|\cdot|$ is not a norm in $\CC^n$, because it is complex valued and hence the function $|z-w|$ is not a metric on $\CC^n$.

We will  consider mainly complex vectors of the form $ z=e^{i\varphi}x $, that is vectors $x\in \RR^n $ rotated in $\CC^n$ by the angle $\varphi$.

For the set $G\subseteq\RR^n$ and the angle $\varphi\in\RR$ we will consider the rotated set defined by
$$e^{i\varphi}G:=\{e^{i\varphi}x:x\in G\}.$$
We will consider mainly the following unions of rotated sets in $\CC^n$:
$$
\widehat{B}_p:=\bigcup_{k=0}^{p-1}e^{\frac{k\pi i}{p}}B\quad\textrm{and}\quad\widehat{S}_p:=\bigcup_{k=0}^{p-1}e^{\frac{k\pi i}{p}}S\quad\textrm{for}
\quad p\in\NN,
$$
where $B$ and $S$ are respectively the unit ball and sphere in $\RR^n$ with a centre at the origin.

The sets
$$ LB:=\{z\in \mathbb{C}^n\colon L(z)<1 \},\qquad
LS:=\{e^{i\varphi }x\colon \varphi \in \mathbb{R},\ x\in S \},$$
where  
$$ L(z)=\sqrt{||z||^2+\sqrt{||z||^4-|z^2|_{\CC}^2}}, $$
with $z^2=z\cdot z$, are called the \emph{Lie ball} and the \emph{Lie sphere}, respectively.

\begin{Uw}
\label{Re:1}
 Observe that the sets $B,\widehat{B}_p, LB$ and $S,\widehat{S}_p,LS$, as the subsets of $\CC^n$, are related by the inclusions
 $$ B=\widehat{B}_1\subset \widehat{B}_p\subset\widehat{B}_{kp}\subset LB\quad\textrm{and}\quad S=\widehat{S}_1\subset \widehat{S}_p
 \subset\widehat{S}_{kp}\subset LS$$
 for every $k,p\in\NN$, $k,p>1$.
\end{Uw}

\section{Polyharmonic polynomials}
In this section we recall the notion of harmonic polynomials and next we extend it to polyharmonic polynomials.

Let $m,p\in \mathbb{\NN}$. We denote by $\mathcal{P}_m(\CC^n)$ the space of all homogeneous polynomials of degree $m$ on $\CC^n$.

We also denote by $ \mathcal{H}_m^p(\CC^n)\subseteq \mathcal{P}_m(\CC^n)$ the space of polynomials on $\CC^n$, which are homogeneous of degree $m$
and are polyharmonic of order~$p$. Observe that in the case $m<2p$ the space $ \mathcal{H}_m^p(\CC^n)$ is equal to $\mathcal{P}_m(\CC^n)$.

In the special case $p=1$ we obtain the space of homogeneous harmonic polynomials of degree $m$,
which is denoted briefly by $\mathcal{H}_m(\CC^n)$.

Let us recall the lemmas relating to the harmonic polynomials.

\begin{Le}[{\cite[Propositions 5.5]{A-B-R}}]
\label{Le:dod_1}
If $m\geq 2$ then 
\begin{equation*}
\mathcal{P}_m(\CC^n)=\mathcal{H}_m(\CC^n)\oplus |x|^2\mathcal{P}_{m-2}(\CC^n),
\end{equation*}
where $\oplus$ denotes the algebraic direct sum,
which means that every element of $\mathcal{P}_m(\CC^n)$ can be uniquely written as the sum of an element of $\mathcal{H}_m(\CC^n)$ and an element of
$|x|^2\mathcal{P}_{m-2}(\CC^n)$.
\end{Le}

\begin{Le}[{\cite[Theorem 5.7]{A-B-R}}]
\label{Le:2}
If $q\in \mathcal{P}_m(\CC^n)$ then there exist unique $q_{m-2k}\in\mathcal{H}_{m-2k}(\CC^n), k=0,1,\dots,[\frac{m}{2}] $  such that
$$q(x)=\sum_{k=0}^{[\frac{m}{2}]}|x|^{2k}q_{m-2k}(x).$$
\end{Le}

\begin{Le}[{\cite[Propositions 5.8]{A-B-R}}]
\label{Le:dod_2}
The spaces $\mathcal{P}_m(\CC^n)$ and $\mathcal{H}_m(\CC^n)$ are finite-dimensional and
\begin{eqnarray*}
 \dim \mathcal{P}_m(\CC^n)&=&\binom{n+m-1}{n-1},\\
 \dim\mathcal{H}_m(\CC^n)&=&\left\{
  \begin{array}{lll}
    \dim \mathcal{P}_m(\CC^n) & \textrm{for} & m=0,1\\
    \dim \mathcal{P}_m(\CC^n)-\dim \mathcal{P}_{m-2}(\CC^n) & \textrm{for} & m\geq 2
  \end{array}
  \right..
\end{eqnarray*}
\end{Le}

To extend the above results to polyharmonic case, we use a following version of the Almansi theorem \cite[Proposition 1.3]{A-C-L}
for homogeneous polynomials.
\begin{St}
\label{F:1}
Let $m\geq 2p $ (resp. $m<2p$). A function $u\in \mathcal{H}_m^p(\CC^n)$ if and only if there exist uniquely determined functions
$ u_k\in  \mathcal{H}_{m-2k}(\CC^n)$,
$k=0,1,\dots,p-1$ (resp. $k=0,1,\dots,[\frac{m}{2}]$)
such that
\begin{gather}
\label{eq:1}
u(x)=\sum_{k=0}^{p-1}|x|^{2k}u_k(x)\qquad\textrm{for}\quad x\in \CC^n
\end{gather}
and respectively
\begin{equation*}
u(x)=\sum_{k=0}^{[\frac{m}{2}]}|x|^{2k}u_k(x)\qquad\textrm{for}\quad x\in \CC^n.
\end{equation*}
\end{St} 
\begin{proof} 
Let $m\geq 2p$. If $u\in  \mathcal{H}_m^p({\CC}^n) $, then by the Almansi theorem \cite[Proposition 1.3]{A-C-L} there are functions $u_k$ each harmonic on
${\CC}^n$ such that the formula (\ref{eq:1}) is valid. From construction of these functions (see proof of \cite[Proposition 1.3]{A-C-L})
it is easy to note that $u_k$ are uniquely determined and homogeneous of degree $m-2k$ for $k=0,1,\dots,p-1$.

If $m<2p$, then the statement is simply a restatement of Lemma \ref{Le:2} because here we have $\mathcal{H}_m^p(\CC^n)=\mathcal{P}_m(\CC^n)$.

The second implication is obvious.
\end{proof}
\begin{Uw} 
\label{agawa}
From now on, unless stated otherwise, we will assume that every $u\in \mathcal{H}_m^p(\CC^n)$ has the expansion (\ref{eq:1}) as in the case
$m\geq 2p$, keeping in mind that $u_k(x)\equiv 0$ for $k>[\frac{m}{2}]$ in the case when $m<2p$.
\end{Uw}

Now we are ready to find the analogues of Lemmas \ref{Le:dod_1}--\ref{Le:dod_2} for polyharmonic polynomials
\begin{St}
\label{F:2}
If $m\geq 2p$ then
\begin{equation*}
\label{eq:8}
\mathcal{P}_m(\CC^n)=\mathcal{H}^p_m(\CC^n) \oplus |x|^{2p}\mathcal{P}_{m-2p}(\CC^n),
\end{equation*}
where $\oplus$ denotes the algebraic direct sum.
\end{St}
\begin{proof}
It is sufficient to use Lemmas \ref{Le:dod_1}--\ref{Le:2} and Proposition \ref{F:1}.
\end{proof}

\begin{St}
\label{Le:5}
If $q\in \mathcal{P}_m(\CC^n)$, then there exist uniquely determined $q_{m-2kp}\in\mathcal{H}^p_{m-2kp}(\CC^n)$, $k=0,1,\dots,[\frac{m}{2p}] $
such that
$$q(x)=\sum_{k=0}^{[\frac{m}{2p}]}|x|^{2kp}q_{m-2kp}(x).$$
\end{St}
\begin{proof}
It is sufficient to apply inductively Proposition \ref{F:2}.
\end{proof}

\begin{St}
\label{Wn:3}
The space $\mathcal{H}^p_m(\CC^n)$ is finite dimensional and
\begin{equation*}
\label{eq:11}
\dim\mathcal{H}^p_m(\CC^n)=\left\{
  \begin{array}{lll}
    \dim \mathcal{P}_m(\CC^n)=\binom{n+m-1}{n-1} & \textrm{for} & m<2p\\
    \dim \mathcal{P}_m(\CC^n)-\dim \mathcal{P}_{m-2p}(\CC^n)=&\\
    =\binom{n+m-1}{n-1}-\binom{n+m-2p-1}{n-1}& \textrm{for} & m\geq 2p
  \end{array}
  \right..
\end{equation*}
\end{St}
\begin{proof}
 It is sufficient to use Lemma \ref{Le:dod_2} and Proposition \ref{F:2}.
\end{proof}

\section{Spherical polyharmonics and decomposition of $L^2(\widehat{S}_p)$}
In this section we introduce the notion of spherical polyharmonics which are a natural generalisation of spherical harmonics.
\begin{Df}
\label{Df:1}
The restriction to the set $\widehat{S}_p:= \bigcup_{k=0}^{p-1}e^{\frac{k\pi i}{p}}S $ of an element of $\mathcal{H}_m^p(\CC^n)$
is called a \emph{spherical polyharmonic of degree $m$ and order $p$}.

The set of spherical polyharmonics is denoted by $ \mathcal{H}_m^p( \widehat{S}_p ) $, so
\begin{equation*}
\mathcal{H}_m^p( \widehat{S}_p ):=\left\{ u|_{\widehat{S}_p}\colon u\in \mathcal{H}_m^p(\CC^n)   \right\}.
\end{equation*}
\end{Df}

\begin{Df}
The spherical polyharmonics of order $1$ are called \emph{spherical harmonics} and their space is denoted by $\mathcal{H}_m(S):=\mathcal{H}_m^1(S)$
(see \cite[Chapter 5]{A-B-R}). Analogously we write $\mathcal{H}_{m}(\CC^n)$ instead of $\mathcal{H}_{m}^1(\CC^n)$.
\end{Df}

Restricting to $\widehat{S}_p$ the functions $u$ and $u_k$ from Proposition \ref{F:1} (see also Remark \ref{agawa}) we conclude that 
\begin{Wn}
\label{Wn:1}
A function $u\in \mathcal{H}_m^p( \widehat{S}_p )$ if and only if there exist uniquely determined spherical harmonics
$ u_k\in  \mathcal{H}_{m-2k}(\widehat{S}_p)$, $k=0,1,\dots,p-1$, such that
\begin{gather}
\label{eq:2}
u(x)=\sum_{k=0}^{p-1}e^{\frac{2jk\pi i}{p}}u_k(x)\quad\textrm{for}\quad x\in e^{\frac{j\pi i}{p}}S,\quad j=0,1,\dots,p-1.
\end{gather}
\end{Wn}

Analogously, since every polynomial is a sum of homogeneous polynomials, by Proposition \ref{Le:5} we obtain
\begin{Wn}
\label{Wn:2}
If $q$ is a polynomial of degree $m$ then the restriction of $q$ to $\widehat{S}_p$ is a sum of spherical polyharmonics of degrees at most
$m$.
\end{Wn} 

We will equip the space of functions on $\widehat{S}_p$ with the inner product being a generalisation of the
standard inner product in $L^2(S)$. Namely we have
\begin{Df}
By $ L^2 ( \widehat{S}_p )  $ we mean the usual Hilbert space of square-integrable functions on $\widehat{S}_p$
with the inner product defined by
\begin{equation}
\label{eq:3}
\left\langle f,g\right\rangle _{ \widehat{S}_p }:=\frac{1}{p} \int\limits_S \sum_{j=0}^{p-1}f(e^{\frac{j\pi i}{p}} \zeta )
\overline{g(e^{\frac{j\pi i}{p}} \zeta )}\,d\sigma (\zeta),
\end{equation}
where $\sigma$ is the normalised surface-area measure on $S$.
\end{Df}

We will generalise the properties of spherical harmonics (see \cite[Section 5]{A-B-R}) to spherical polyharmonics.

The main aim of this section
is to show the natural orthogonal decomposition of $L^2( \widehat{S}_p )$ into the spaces of spherical polyharmonics,
what is a polyharmonic version of \cite[Theorem 5.12]{A-B-R}. To this end we recall the definition of the direct sum  of Hilbert spaces
(see \cite[p.~81]{A-B-R}).

\begin{Df}
\label{Df:2}
Let $H$ be a Hilbert space. We say that $H$ is the \emph{direct sum} of spaces $H_m$ and we write $H=\bigoplus_{m=0}^{\infty} H_m$  if the following
conditions are satisfied:
\begin{enumerate}
\item[i)] $H_m$ is a closed subspace of $H$ for every $m$.
\item[ii)] $H_m$ is orthogonal to $H_k$ if $m\neq k$.
\item[iii)] For every $ x \in H$ there exist $ x_m \in H_m$  such that $ x=x_0+x_1+x_2+\dots$, where the sum is converging in the norm of $H$.
\end{enumerate}
\end{Df}

\begin{Uw}
\label{Uw:1}
When the conditions i) and ii) hold, then the condition iii) holds if and only if the set $\Span\bigcup_{m=0}^{\infty} H_m$ is dense in $H$.
\end{Uw}
Let us recall the following orthogonal property of spherical harmonics.
\begin{Le}[{\cite[Proposition 5.9]{A-B-R}}]
\label{Le:1}
If  $ m\neq l $ then $ \mathcal{H}_m(S)$ is orthogonal to $ \mathcal{H}_l(S)$ in $L^2(S)$ with the inner product defined by
\begin{equation*}
 \left\langle f,g \right\rangle_S:=\int_S f(\zeta) \overline{g(\zeta)}\,d\sigma(\zeta).
\end{equation*}

Moreover, if we restrict all functions to $S$, then the decomposition given in Lemma \ref{Le:dod_1} is an orthogonal decomposition
with respect to the inner product on $L^2(S)$.
\end{Le}

We will find the version of Lemma \ref{Le:1} for spherical polyharmonics.
\begin{St}
\label{Le:4}
If  $ m\neq l $ then $ \mathcal{H}^p_m(S)$ is orthogonal to $ \mathcal{H}^p_l(S)$ in $L^2(\widehat{S}_p)$.
Moreover, if we restrict all functions to $\widehat{S}_p$, then the decomposition given in Proposition \ref{F:2} is  orthogonal
in $L^2(\widehat{S}_p)$.
\end{St}
\begin{proof}
Let $ u \in \mathcal{H}_m^p ( \widehat{S}_p )$ and $ v \in \mathcal{H}_l^p ( \widehat{S}_p )$for some $ m\neq l $.
Then by Corollary \ref{Wn:1} there exist $ u_k\in  \mathcal{H}_{m-2k}(S)$ and $ v_k\in  \mathcal{H}_{l-2k}(S)$, $k=0,1,\dots,p-1$, such that
\begin{gather*}
u(x)=\sum_{k=0}^{p-1}e^{\frac{2jk\pi i}{p}}u_k(x)\quad\textrm{and}\quad 
v(x)=\sum_{k=0}^{p-1}e^{\frac{2jk\pi i}{p}}v_k(x) 
\end{gather*}
for $ x\in e^{\frac{j\pi i}{p}}S$, $j=0,1,\dots,p-1$.
Hence
 \begin{equation*}
\left\langle u,v\right\rangle_{ \widehat{S}_p } =  \frac{1}{p} \int\limits_S \sum_{j=0}^{p-1}\left[ \sum_{k=0}^{p-1}e^{\frac{2j k\pi i}{p}}u_k(e^{\frac{j\pi i}{p}} \zeta )
 \overline{\sum_{k=0}^{p-1}e^{\frac{2jk\pi i}{p}}v_k(e^{\frac{j\pi i}{p}} \zeta )} \right] d\sigma (\zeta).
\end{equation*}
By the homogeneity of spherical harmonics we have
\begin{gather*}
u_k(e^{\frac{j\pi i}{p}} \zeta )=e^{\frac{(m-2k)j\pi i}{p}}u_k(\zeta),\quad 
v_k(e^{\frac{j\pi i}{p}} \zeta )=e^{\frac{(l-2k)j\pi i}{p}} v_k(\zeta)
\end{gather*}
for $ \zeta \in S$ and $j,k=0,1,\dots,p-1$. So
\begin{eqnarray*}
\left\langle u,v\right\rangle_{ \widehat{S}_p } & =  & \frac{1}{p} \int\limits_S \sum_{j=0}^{p-1}\left[ \sum_{k=0}^{p-1}e^{\frac{mj\pi i}{p}}
u_k( \zeta )\sum_{k=0}^{p-1}e^{\frac{-lj\pi i}{p}} \overline{v_k( \zeta )} \right] d\sigma (\zeta) \\
& = & \frac{1}{p} \sum_{j=0}^{p-1}e^{\frac{(m-l)j\pi i}{p}}  \sum_{0\leq   \alpha  , \beta \leq p-1}\int\limits_S u_{\alpha}( \zeta )
\overline{v_{\beta} ( \zeta )}  d\sigma (\zeta) \\
 & = & \frac{1}{p} \sum_{j=0}^{p-1}e^{\frac{(m-l)j\pi i}{p}}   \sum_{0\leq   \alpha  , \beta \leq p-1}
 \left\langle u_{\alpha},v_{\beta}\right\rangle_S. 
\end{eqnarray*}
By Lemma \ref{Le:1}  $ \left\langle u_{\alpha},v_{\alpha}\right\rangle_S=0 $, because $m-2\alpha \neq l-2\alpha$. Therefore
\begin{eqnarray*}
 \left\langle u,v\right\rangle_{ \widehat{S}_p } & =  & 
\frac{1}{p} \sum_{j=0}^{p-1}e^{\frac{(m-l)j\pi i}{p}}   \sum_{\substack{0\leq   \alpha  , \beta \leq p-1 \\ \alpha \neq  \beta}}
\left\langle u_{\alpha},v_{\beta}\right\rangle_S.
\end{eqnarray*}
Let us assume that $m-l$ is an odd integer number. Then obviously $ \left\langle u_\alpha,v_\beta\right\rangle_S=0 $ for any 
$ \alpha,\beta=0,1,\dots,p-1 $ by Lemma \ref{Le:1}. Hence $  \left\langle u,v\right\rangle_{ \widehat{S}_p }=0.$ 
Let us assume now that $ m-l $ is an even integer number, that is there exist  $ \gamma \in \ZZ $ such that $m-l=2\gamma$, so
\begin{eqnarray*}
 \left\langle u,v\right\rangle_{ \widehat{S}_p } &  =  & \frac{1-e^{2\gamma \pi i}}{p( 1-e^{\frac{2\gamma \pi i}{p}}) }  \sum_{\substack{0\leq   \alpha  , \beta \leq p-1 \\ \alpha \neq  \beta}}
\left\langle u_{\alpha},v_{\beta}\right\rangle_S =0,
\end{eqnarray*}
as desired.

If  $v\in|x|^{2p}\mathcal{P}_{m-2p}(\CC^n)$ then $v|_{\widehat{S}_p}\in\mathcal{P}_{m-2p}(\widehat{S}_p)$, where
$\mathcal{P}_{m-2p}(\widehat{S}_p):=\{u|_{\widehat{S}_p}\colon u\in\mathcal{P}_{m-2p}(\CC^n)\}$.
So, by Corollary \ref{Wn:2}, $v|_{\widehat{S}_p}$ is a sum of spherical polyharmonics of degrees at most $m-2p$.
Hence, by the first part of the proof $\langle v|_{\widehat{S}_p},w|_{\widehat{S}_p}\rangle_{\widehat{S}_p}=0$
for any element $w\in\mathcal{H}_m^p(\CC^n)$. 
\end{proof}

\begin{Uw}
\label{Uw:2}
In the case when $u,v\in\mathcal{H}_m^p ( \widehat{S}_p )$, their 
inner product in $L^2(\widehat{S}_p)$ reduces to the ordinary inner product in $L^2(S)$:
\begin{equation}
\label{eq:12}
\left\langle u,v\right\rangle_{ \widehat{S}_p }=\frac{1}{p} \int\limits_S \sum_{j=0}^{p-1}e^{\frac{mj\pi i}{p}}u( \zeta )
e^{\frac{-mj\pi i}{p}}\overline{v( \zeta )}\,d\sigma (\zeta)=\left\langle u,v\right\rangle_S.
\end{equation}
\end{Uw}

Next we prove
\begin{St}
\label{Le:6}
The linear span of $\bigcup_{m=0}^{\infty}\mathcal{H}^p_m(\widehat{S}_p)$ is a dense subset of $C(\widehat{S}_p)$ with respect to the supremum norm.
\end{St}
\begin{proof}
Let us assume that $f\in C(\widehat{S}_p)$. It means that
\begin{gather*}
f(x)=u_j(x)\quad\textrm{for}\quad x\in e^\frac{j\pi i}{p}S,\quad j=0,1,\dots,p-1,
\end{gather*}
where $u_j$ is a continuous function on the set  $e^\frac{j\pi i}{p}S$. By the Stone-Weierstrass theorem (see \cite[Theorem 7.26]{R})
there exist sequences of polynomials
$(q_m^j)_{m\in\NN}$ such that $q_m^j$ is of degree $m$ and $q_m^j\rightrightarrows u_j$ uniformly on $e^\frac{j\pi i}{p}S$ as $m\to\infty$ for
$j=0,1,\dots,p-1$. So, if we put
\begin{equation*}
q_m(x):=\frac{1}{p}( \sum_{k=0}^{p-1}q^k_m(x)+|x|^2\sum_{k=0}^{p-1}e^\frac{-2k\pi i}{p}q^k_m(x)+\cdots +|x|^{2(p-1)}\sum_{k=0}^{p-1}e^\frac{-2k(p-1)\pi i}{p}q^k_m(x) )
\end{equation*}
then $q_m\rightrightarrows f$ uniformly on $\widehat{S}_p$, with $q_m\in\bigcup_{k=0}^{m+2(p-1)}\mathcal{H}^p_k(\widehat{S}_p)$.
It means that every continuous function on $\widehat{S}_p$ can be approximated uniformly by a sequence of spherical polyharmonics, as desired.
\end{proof} 

Now we are ready to prove the main result of this section
\begin{Tw}
\label{Tw:1}
\begin{equation*}
L^2 ( \widehat{S}_p )=\bigoplus_{m=0}^{\infty} \mathcal{H}_m^p ( \widehat{S}_p ).
\end{equation*}
\end{Tw}
\begin{proof}
We have to show that the conditions in Definition \ref{Df:2} are satisfied.
The first one holds because  $\mathcal{H}^p_m(\widehat{S}_p)$ is finite dimensional by Proposition \ref{Wn:3}  and hence is closed.
The second condition holds by Proposition \ref{Le:4}.

So it is sufficient to show the density of 
$\Span\bigcup_{m=0}^{\infty} \mathcal{H}_m^p ( \widehat{S}_p )$ in $L^2( \widehat{S}_p )$. 
By Proposition \ref{Le:5}, every polynomial on $\widehat{S}_p$ can be written as a finite sum
of elements of $\bigcup_{m=0}^{\infty} \mathcal{H}_m^p ( \widehat{S}_p )$.
By Proposition \ref{Le:6} we know that the set of polynomials on $\widehat{S}_p$ is a dense subset of the set of continuous functions
$C( \widehat{S}_p )$. 
On the other hand, by Proposition \ref{Le:5} every polynomial on $\widehat{S}_p$ can be written as a finite sum
of elements of $\bigcup_{m=0}^{\infty} \mathcal{H}_m^p ( \widehat{S}_p )$.
Hence the family of finite sums of spherical polyharmonics is dense
in $C( \widehat{S}_p )$ with respect to the supremum norm. It means that
$\Span\bigcup_{m=0}^{\infty} \mathcal{H}_m^p ( \widehat{S}_p )$
is dense in $C( \widehat{S}_p )$.
Additionally $C( \widehat{S}_p )$ is dense in $L^2( \widehat{S}_p )$.
Therefore the set of finite sums of spherical polyharmonics is dense in  $L^2( \widehat{S}_p )$. In consequence 
$\Span\bigcup_{m=0}^{\infty} \mathcal{H}_m^p ( \widehat{S}_p )$ is dense in  $L^2( \widehat{S}_p )$, as desired.
\end{proof}

\section{Zonal polyharmonics}
Let us consider $\mathcal{H}_m^p (\widehat{S}_p)$ as a space  with the inner product (\ref{eq:3}) induced
from $L^2( \widehat{S}_p )$ (we can do that because $\mathcal{H}_m^p (\widehat{S}_p)$ is the closed subspace of the
Hilbert space $L^2( \widehat{S}_p )$). 
Analogously to zonal harmonics we will introduce so called zonal polyharmonics (see \cite[p.~94]{A-B-R}).

Let $ \eta \in \widehat{S}_p$ be a fixed point. Let us consider the linear functional
$\Lambda_{\eta}\colon \mathcal{H}_m^p (\widehat{S}_p) \longrightarrow {\CC} $ defined as 
\begin{gather*}
\Lambda_{\eta} (q)=q(\eta)\quad\textrm{for}\quad q\in \mathcal{H}_m^p (\widehat{S}_p).
\end{gather*}
Since $\mathcal{H}_m^p (\widehat{S}_p)$ is a finite dimensional inner-product space, it is a self-dual Hilbert space.
Hence there exists
a unique $Z_m^p(\cdot,\eta)\in \mathcal{H}_m^p (\widehat{S}_p)$ such that 
\begin{equation}
\label{eq:13}
q(\eta)=\left\langle q,Z^p_m(\cdot,\eta)\right\rangle_{  \widehat{S}_p }\quad\textrm{for every}\quad q\in
\mathcal{H}^p_m(\widehat{S}_p).
\end{equation}

\begin{Df} 
\label{Df:3}
The function $Z^p_m(\cdot,\eta)$ satisfying (\ref{eq:13}) is called a \emph{zonal polyharmonic}
of degree $m$ and of order $p$ with a pole $\eta$. 
\end{Df}

\begin{Uw}
\label{Uw:3}
Zonal polyharmonics of order $p=1$ are called \emph{zonal harmonics}. Throughout this paper we will denote them by
$Z_m(\cdot,\eta)$ instead of $Z^1_m(\cdot,\eta)$ for $\eta \in S$.
\end{Uw}

Let us recall some properties of zonal harmonics. We will use them in the subsequent considerations.
\begin{Le}[{\cite[Proposition 5.27]{A-B-R}}]
\label{Le:zonal_1}
Suppose $\zeta, \eta \in S$. Then:
\begin{enumerate}
\item[(a)] $Z_m(\zeta,\eta)\in \RR$.
\item[(b)] $Z_m(\zeta,\eta)=Z_m(\eta,\zeta)$.
\item[(c)] $Z_m(\eta,\eta)=\dim\mathcal{H}_m(\CC^n)$.
\item[(d)] $|Z_m(\zeta, \eta)|\leq \dim\mathcal{H}_m(\CC^n)$.
\end{enumerate}
\end{Le}

Before we proceed to properties of zonal harmonics, at first we extend the definition of zonal harmonics
$Z_m(\cdot,\eta)$ to the case when $\eta\in\widehat{S}_p$ taking
\begin{gather*}
Z_m(e^{\varphi i}\zeta,e^{\psi i}\eta):=e^{m(\varphi-\psi) i}Z_m(\zeta,\eta)\qquad\textrm{for}\quad\zeta,\eta \in S. 
\end{gather*}
It is obvious that this extension preserves the degree of homogeneity as well as the harmonicity. Also we have for $\zeta,\eta \in S$ that
\begin{gather*}
Z_m(e^{\varphi i}\zeta,e^{\psi i}\eta)=Z_m(e^{-\psi i}\eta,e^{-\varphi i}\zeta)\quad\textrm{and}\quad
\overline{Z_m(e^{\varphi i}\zeta,e^{\psi i}\eta)}=Z_m(e^{\psi i}\eta,e^{\varphi i}\zeta).
\end{gather*}
\begin{St}
\label{F:3}
Let $\zeta,\eta \in \widehat{S}_p$ and $a\in \CC$. Then:
\begin{enumerate}
\item[(a)] $\overline{Z_m^p(\zeta,\eta)}=Z^p_m(\eta,\zeta)$.
\item[(b)] $Z_m^p(\zeta,\zeta)\in \RR$.
\item[(c)] $Z_m^p(a\zeta,\eta)=Z^p_m(\zeta,\overline{a} \eta)$.
\end{enumerate}
\end{St} 
\begin{proof} 
To prove (a), let $\{e_1,e_2,\dots,e_{h^p_m}\}$ be an orthonormal basis of $\mathcal{H}_m^p (\widehat{S}_p)$,
where $h_m^p:=\dim\mathcal{H}_m^p (\widehat{S}_p)$.
Then for every $\zeta,\eta \in \widehat{S}_p$ we have
\begin{eqnarray*}
Z^p_m(\zeta,\eta)&=&\sum_{k=1}^{h^p_m}\left\langle Z^p_m(\cdot,\eta),e_k\right\rangle_{\widehat{S}_p} e_k(\zeta)=
\sum_{k=1}^{h^p_m} \overline{\left\langle  e_k,Z^p_m(\cdot,\eta)\right\rangle}_{\widehat{S}_p} e_k(\zeta)\\
&=& \sum_{k=1}^{h^p_m}\overline{e_k(\eta)}e_k(\zeta)=\overline{\sum_{k=1}^{h^p_m}\overline{e_k(\zeta)}e_k(\eta)}=\overline{Z^p_m(\eta,\zeta)},
\end{eqnarray*}
which gives the first claim. The next claims follow simply from the first one. 
\end{proof}

Now we will prove the result which gives the connection between zonal polyharmonics and extensions of zonal harmonics.
\begin{Tw}
\label{Tw:2}
Let $\zeta,\eta \in \widehat{S}_p$, then
\begin{equation*}
\label{eq:14}
Z_m^p(\zeta ,\eta)=\sum_{k=0}^{p-1}|\zeta|^{2k}|\overline{\eta}|^{2k}Z_{m-2k}(\zeta,\eta).
\end{equation*}
\end{Tw}
\begin{proof}
Let $\eta\in e^{\frac{j\pi i}{p}}S$ be a fixed point, so $\eta:=e^{\frac{j\pi i}{p}}\xi$ for some $\xi\in S$. We assume also that $q\in \mathcal{H}_m^p (\widehat{S}_p)$. Then by (\ref{eq:12}) and (\ref{eq:13}) we have
\begin{equation}
\label{eq:15}
q(\eta)=\left\langle q,Z^p_m(\cdot,\eta)\right\rangle_{  \widehat{S}_p }=\left\langle q,Z^p_m(\cdot,\eta)\right\rangle_S.
\end{equation} 
On the other hand by Corollary \ref{Wn:1} there are spherical harmonics $q_k$ of degree   $m-2k$ for $k=0,1,\dots,p-1$, such that
$q(\xi)=\sum_{k=0}^{p-1}q_k(\xi)$.
Hence there are $Z_{m-2k}(\cdot,\xi)$ such that
$q_k(\xi)=\left\langle q_k,Z_{m-2k}(\cdot,\xi)\right\rangle_S.$
From the above considerations and from the appropriate properties of zonal harmonics we have
\begin{eqnarray*}
q(\eta) & = & e^{\frac{jm\pi i}{p}}q(\xi)=e^{\frac{jm\pi i}{p}}\sum_{k=0}^{p-1}\int_S q_k(\zeta)Z_{m-2k}(\zeta,\xi)d\sigma(\zeta) \\
& = & \sum_{k=0}^{p-1}e^{\frac{2jk\pi i}{p}}\int_S q_k(\zeta)Z_{m-2k}(\zeta,e^{\frac{-j\pi i}{p}}\xi)d\sigma(\zeta)\\
& = & \sum_{k=0}^{p-1}e^{\frac{2jk\pi i}{p}}\int_S q_k(\zeta)\overline{Z_{m-2k} (\zeta,\eta)}d\sigma(\zeta)\\
& = & \sum_{k=0}^{p-1}e^{\frac{2jk\pi i}{p}}\left\langle q,Z_{m-2k}(\cdot,\eta)\right\rangle_S
=\left\langle q,\sum_{k=0}^{p-1}e^{\frac{-2jk\pi i}{p}}Z_{m-2k}(\cdot,\eta)\right\rangle_S.
\end{eqnarray*}
Comparing the last equality with (\ref{eq:15}) we obtain
\begin{gather*}
Z^p_m(\zeta,\eta)=\sum_{k=0}^{p-1}e^{\frac{-2jk\pi i}{p}}Z_{m-2k}(\zeta,\eta)\quad\textrm{for}\quad\zeta\in S.
\end{gather*}
Further we have
$$Z^p_m(e^{\frac{l\pi i}{p}}\zeta,\eta)=e^{\frac{ml\pi i}{p}}Z^p_m(\zeta,\eta)
=\sum_{k=0}^{p-1}e^{\frac{-2jk\pi i}{p}}e^{\frac{2lk\pi i}{p}}Z_{m-2k}(e^{\frac{l\pi i}{p}}\zeta,\eta).$$
Therefore if $\zeta,\eta\in \widehat{S}_p$ then
$$Z_m^p(\zeta ,\eta)=\sum_{k=0}^{p-1}|\zeta|^{2k}|\overline{\eta}|^{2k}Z_{m-2k}(\zeta,\eta),$$
as desired.
\end{proof}
\begin{Uw}
\label{re:Z}
 As in Remark \ref{agawa}, we keep in mind that $Z_{m-2k}(\zeta,\eta)\equiv 0$ for $m-2k<0$ in the case when $m<2p$.
\end{Uw}

\begin{St}
\label{Wn:4}
If $\eta\in\widehat{S}_p$ then 
\begin{enumerate}
\item[(a)] $Z^p_m(\eta,\eta)=\dim\mathcal{H}_m^p(\CC^n)$.
\item[(b)] $|Z_m^p(\zeta,\eta)|_{\CC} \leq \dim\mathcal{H}_m^p(\CC^n)$ for all $\zeta\in \widehat{S}_p$.
\end{enumerate}
\end{St}
\begin{proof}
To prove (a), suppose that  $\eta=e^{\frac{j\pi i}{p}}\zeta$ with $\zeta\in S$. By the last theorem we can write
$$Z^p_m(\eta,\eta)=Z^p_m(\zeta,\zeta)=\sum_{k=0}^{p-1}Z_{m-2k}(\zeta,\zeta)=\sum_{k=0}^{p-1}\dim\mathcal{H}_{m-2k}(\CC^n)=
\dim\mathcal{H}_m^p(\CC^n),$$
where in the third and fourth equality we use Lemmas \ref{Le:dod_2}, \ref{Le:zonal_1}, and Proposition \ref{Wn:3}.

Formula (b) is given by Lemma \ref{Le:zonal_1}. 
\end{proof}

Since $Z_m^p(\cdot ,\eta)\in \mathcal{H}^p_m(\widehat{S}_p)$, we can extend zonal polyharmonics to $\CC^n$. Indeed, for $x\in \CC^n$ we can write
\begin{equation*}
Z_m^p(x ,\eta) = |x|^mZ_m^p(\frac{x}{|x|} ,\eta) = |x|^m \sum_{k=0}^{p-1} \left| \frac{x}{|x|}\right|^{2k}|\overline{\eta}|^{2k}
Z_{m-2k}(\frac{x}{|x|},\eta),
\end{equation*}
so
\begin{gather}
\label{eq:16}
Z^p_m(x,\eta)=\sum_{k=0}^{p-1} |x|^{2k}|\overline{\eta}| ^{2k}Z_{m-2k}(x,\eta)\quad\textrm{for}\quad x\in\CC^n.
\end{gather}
Analogously we extend spherical polyharmonics to $\CC^n$. For $x\in \CC^n$ we have
\begin{equation*}
q(x) = |x|^m q(\frac{x}{|x|})=|x|^m\int_S q(\zeta)Z_m^p(\frac{x}{|x|} ,\zeta)\,d\sigma(\zeta).
\end{equation*}
Hence
\begin{gather}
\label{eq:17}
q(x)=\int_S q(\zeta)Z^p_m(x,\zeta)d\sigma(\zeta)\quad\textrm{for}\quad x\in\CC^n.
\end{gather}

The spherical polyharmonics allow us to find the orthogonal decomposition from Theorem \ref{Tw:1} for a given function
$f\in L^2(\widehat{S}_p)$.
Namely, we have
\begin{Tw}
\label{Tw:3}
Let $f\in L^2 ( \widehat{S}_p )$. Then
\begin{gather*}
f(\eta)=\sum_{m=0}^{\infty}\left\langle f,Z^p_m(\cdot,\eta)\right\rangle_{\widehat{S}_p}\quad\textrm{in}\quad L^2(\widehat{S}_p).
\end{gather*}
\end{Tw}
\begin{proof}
Indeed, if $f\in  L^2 ( \widehat{S}_p )$, then by Theorem \ref{Tw:1} there exist $q_m \in \mathcal{H}_m^p(\widehat{S}_p)$
such that
\begin{eqnarray}
\nonumber f(\eta)&=&\sum_{m=0}^{\infty}q_m(\eta)=\sum_{m=0}^{\infty}\left\langle q_m,Z^p_m(\cdot,\eta)\right\rangle_{\widehat{S}_p}\\
\nonumber &=&\sum_{m=0}^{\infty}\left\langle \sum_{k=0}^{\infty}q_k,Z^p_m(\cdot,\eta)\right\rangle_{\widehat{S}_p}=
\sum_{m=0}^{\infty}\left\langle f,Z^p_m(\cdot,\eta)\right\rangle_{\widehat{S}_p}.
\end{eqnarray}
\end{proof}

\section{Poisson kernel for the union of rotated balls}
We will find a Poisson kernel for polyharmonic functions on the $\widehat{B}_p$. To this end let us accept the following definitions
\begin{Df}
\label{Df:4}
The function $P_p\colon(\widehat{B}_p\times \widehat{S}_p) \cup(\widehat{S}_p \times \widehat{B}_p)\to\CC$ is called a \emph{Poisson kernel
for $\widehat{B}_p$} provided for every polyharmonic function $u$ on $\widehat{B}_p$ which is continuous on $\widehat{B}_p\cup\widehat{S}_p$ and for each
$x\in \widehat{B}_p$ holds
\begin{gather*}
u(x)=\left\langle u, P_p(\cdot,x) \right\rangle_{\widehat{S}_p}=\frac{1}{p}\sum_{j=0}^{p-1}\int_S u(e^{\frac{j\pi i}{p}}\zeta)
\overline{P_p(e^{\frac{j\pi i}{p}}\zeta ,x )}\, d\sigma(\zeta).
\end{gather*}
\end{Df}

\begin{Df}
\label{Df:5} 
Let $f\in C(\widehat{S}_p)$. The function defined for $x\in \widehat{B}_p$ by
\begin{equation}
\label{eq:poisson}
P_p[f](x):=\left\langle f, P_p(\cdot,x) \right\rangle_{\widehat{S}_p}=\frac{1}{p}\sum_{j=0}^{p-1}\int_S f(e^{\frac{j\pi i}{p}}\zeta)
\overline{P_p(e^{\frac{j\pi i}{p}}\zeta ,x )}\, d\sigma(\zeta)
\end{equation}
is called a \emph{Poisson integral for $f$}.
\end{Df}

\begin{Uw}
\label{uwaga3}
Note that for $p=1 $ the above definitions are well known. The function $P(x,\zeta):=P_1(x,\zeta)$
is the classical Poisson kernel for the usual euclidean real ball (see \cite[Proposition 5.31]{A-B-R} and \cite[Theorem 3]{F-M})
\begin{equation}
\label{eq:18}
P(x,\zeta)=\sum_{m=0}^{\infty}Z_m(x,\zeta) = \frac{1-|x|^2|\overline{\zeta}|^2}{(x^2\overline{\zeta}^2-2x\cdot\overline{\zeta}+1)^{n/2}}.
\end{equation}
\end{Uw}

\begin{St} 
\label{Tw:4}
If $f$ is a polynomial of degree $m$ on $\CC^n$, then $P_p[f]$ is a polynomial of degree at most $m$ and
\begin{equation*}
\label{eq:19}
P_p[f](x)=\sum_{k=0}^{m}\left\langle f,Z^p_k(\cdot,x)\right\rangle_{\widehat{S}_p}\quad\textrm{for every}\quad x\in \widehat{B}_p.
\end{equation*} 
\end{St}
\begin{proof}
By Corollary \ref{Wn:2} there exist spherical polyharmonics $g_k$ of degree $k$ such that
\begin{gather*}
f(x)=\sum_{k=0}^{m}g_k(x)\quad\textrm{for}\quad x\in \widehat{S}_p.
\end{gather*}
We extend functions $g_k$ to the sets $\widehat{B}_p$ (they will be still polyharmonic and homogeneous):
\begin{gather*}
P_p[f](x)=\sum_{k=0}^{m}g_k(x)\quad\textrm{for}\quad x\in \widehat{B}_p.
\end{gather*}
From the last equality we see that $P_p[f]$ is a polynomial of degree at most $m$. Using (\ref{eq:17}) we get
$$g_k(x)=\int_S Z^p_k(x,\zeta)g_k(\zeta)\,d\sigma(\zeta).$$
Because of orthogonality of polyharmonics of different degrees and by the above considerations we conclude that
\begin{eqnarray*}
P_p[f](x)&=&\sum_{k=0}^{m}\int_S Z^p_k(x,\zeta)g_k(\zeta)d\zeta=\sum_{k=0}^{m}\int_S Z^p_k(x,\zeta)f(\zeta)d\zeta\\
&=& \sum_{k=0}^{m}\left\langle f,Z^p_k(\cdot,x)\right\rangle_S=\sum_{k=0}^{m}\left\langle f,Z^p_k(\cdot,x)\right\rangle_{\widehat{S}_p}.
\end{eqnarray*}
\end{proof}

In the next theorem we will need 
\begin{Le}[{\cite[p.~99]{A-B-R}}]
\label{Le:7}
Let $\zeta\in S$. Then there exist a constant $C>0$ such that
\begin{gather*}
\label{eq:20}
|Z_m(x,\zeta)|_{\CC}\leq Cm^{n-2}|x|_{\CC}^n\quad\textrm{for every}\quad x\in \CC^n.
\end{gather*}
\end{Le}

Now we are ready to give the connection between the Poisson kernel for $\widehat{B}_p$ and the zonal polyharmonics,
what allows us to find an explicit formula for the Poisson kernel. Namely we have
\begin{Tw}
\label{Wn:7}
The Poisson kernel has the expansion
\begin{equation}
\label{eq:21}
P_p(x,\zeta)=\sum_{m=0}^{\infty}Z^p_m(x,\zeta)=\frac{1-|x|^{2p}}{(x^2\overline{\zeta}^2-2x\cdot\overline{\zeta}+1)^{n/2}}\quad\textrm{for}
\quad x\in \widehat{B}_p,\,\, \zeta \in \widehat{S}_p.
\end{equation}
The series converges absolutely and uniformly on $K\times \widehat{S}_p$, where $K$ is a compact subset of $\widehat{B}_p$.
\end{Tw}
\begin{proof}
Let $q\in \mathcal{H}^p_m(\widehat{S}_p)$. By the density of sums of spherical polyharmonics in $L^2(\widehat{S}_p)$
it is sufficient to show that
$$\left\langle q,P_p(\cdot,x)\right\rangle_{  \widehat{S}_p }=\left\langle q,\sum_{k=0}^{\infty}Z^p_k(\cdot,x)\right\rangle_{\widehat{S}_p}.$$
It holds, because
$$\left\langle q,P_p(\cdot,x)\right\rangle_{  \widehat{S}_p }=q(x)=\left\langle q,Z^p_m(\cdot,x)\right\rangle_{  \widehat{S}_p }=
\left\langle q,\sum_{k=0}^{\infty}Z^p_k(\cdot,x)\right\rangle_{  \widehat{S}_p}.
$$
Now we will show that the series in (\ref{eq:21}) is absolute convergent. By (\ref{eq:16}) and Lemma \ref{Le:7} we estimate
\begin{eqnarray*}
|Z^p_m(x,\zeta)|_{\CC} & \leq & \sum_{k=0}^{p-1}\left||x|^{2k}\right|_{\CC}\left||\overline{\zeta}|^{2k}\right|_{\CC}\left|Z_{m-2k}(x,\zeta)\right|_{\CC}\\
& \leq & C|x|^{m}_{\CC}\sum_{k=0}^{p-1}m^{n-2}\leq Cpm^{n-2}|x|^m_{\CC}.
\end{eqnarray*}

Hence for every compact subset $K$ of $\widehat{B}_p$ we conclude that 
$$\max_{(x,\zeta)\in K\times \widehat{S}_p}\sum_{m=0}^{\infty}|Z^p_m(x,\zeta)|_{\CC}\leq Cp\max_{(x,\zeta)\in K\times \widehat{S}_p}
\sum_{m=0}^{\infty}m^{n-2}|x|_{\CC}^m<\infty,
$$
which means that the series in (\ref{eq:21}) converges absolutely and uniformly, as desired.
 
To prove the second equality in (\ref{eq:21}), observe that by (\ref{eq:16}) and by the convergence of the series in (\ref{eq:21}) we have
$$
\sum_{m=0}^{\infty}Z^p_m(x,\zeta)=\sum_{m=0}^{\infty}\sum_{k=0}^{p-1}|x|^{2k}|\overline{\zeta}|^{2k}Z_{m-2k}(x,\zeta)=
\sum_{k=0}^{p-1}|x|^{2k}|\overline{\zeta}|^{2k}\sum_{m=0}^{\infty}Z_{m-2k}(x,\zeta).
$$
Moreover, by Remark \ref{re:Z}, $Z_{m-2k}(x,\zeta)\equiv 0$ for $m<2k$. Hence we get 
\begin{equation*}
\sum_{m=0}^{\infty}Z^p_m(x,\zeta)=\sum_{k=0}^{p-1}|x|^{2k}|\overline{\zeta}|^{2k}\sum_{m=2k}^{\infty}Z_{m-2k}(x,\zeta)=
\sum_{k=0}^{p-1}|x|^{2k}|\overline{\zeta}|^{2k}\sum_{m=0}^{\infty}Z_{m}(x,\zeta)
\end{equation*}
for every $x\in \widehat{B}_p$ and $\zeta \in \widehat{S}_p$. From the above and by (\ref{eq:18}) we have 
\begin{eqnarray*}
P_p(x,\zeta)&=&(1+|x|^2|\overline{\zeta}|^2+\dots+|x|^{2(p-1)}|\overline{\zeta}|^{2(p-1)})
\frac{1-|x|^{2}|\overline{\zeta}|^2}{(x^2\overline{\zeta}^2-2x\cdot\overline{\zeta}+1)^{n/2}}\\
&=&\frac{1-|x|^{2p}}{(x^2\overline{\zeta}^2-2x\cdot\overline{\zeta}+1)^{n/2}},
\end{eqnarray*}
because $|\overline{\zeta}|^{2p}=1$ for $\zeta \in \widehat{S}_p$. 
\end{proof}

In the next proposition we will collect the properties of the Poisson kernel for $\widehat{B}_p$.
\begin{St}
\label{Tw:5} 
The Poisson kernel $P_p$ has the following properties:
\begin{itemize}
\item[(a)] $\overline{P_p(\zeta,x)}=P_p(x,\zeta)$ for every $x\in \widehat{B}_p$ and every $ \zeta\in \widehat{S}_p$.
\item[(b)] For every $x\in \widehat{B}_p$ and $ \zeta\in S$, and for $j=0,1,\dots,p-1$,
$$\overline{P_p(e^{\frac{j\pi i}{p}}\zeta,x)}=\frac{1-|x|^{2p}}{|e^{\frac{-j\pi i}{p}}x-\zeta|^n}.$$
\item[(c)] $P_p(\cdot,\zeta)$ is polyharmonic of order $p$ on $\widehat{B}_p$ for all $\zeta \in \widehat{S}_p$. 
\item[(d)] For every $x\in \widehat{B}_p$ and for $k=0,1,\dots,p-1$, 
$$\int_S P_p(e^{\frac{-k\pi i}{p}}x,\zeta)d\sigma(\zeta)=1+e^{\frac{-2k\pi i}{p}}|x|^{2}+\cdots+e^{\frac{-2(p-1)k\pi i}{p}}|x|^{2(p-1)}.$$
\item[(e)] For every $x\in \widehat{B}_p$, 
$$\frac{1}{p}\sum_{k=0}^{p-1}\int_S P_p(e^{\frac{-k\pi i}{p}}x,\zeta)d\sigma(\zeta)=1.$$ 
\item[(f)] If $\zeta\in e^{\frac{j\pi i}{p}}S$, then $P_p(x,\zeta)>0$ for any $x\in e^{\frac{j\pi i}{p}}B$, $j=0,1,\dots,p-1$.
\item[(g)] For every $\eta \in \widehat{S}_p$ and every $\delta>0$,
\begin{gather*}
\sum_{k=0}^{p-1}\int_{\|\zeta-e^{\frac{-k\pi i}{p}}\eta\|>\delta} \left| P_p(e^{\frac{-k\pi i}{p}}x,\zeta) \right|_{\CC} d\sigma(\zeta) \rightarrow 0\quad
\textrm{as}\quad x\rightarrow \eta,\quad x\in\widehat{B}_p.
\end{gather*}
\end{itemize}
\end{St} 
\begin{proof} 
Property (a) follows from Theorem \ref{Wn:7} and Proposition \ref{F:3}.

To prove property (b), let us suppose that $x\in \widehat{B}_p$ and $\zeta \in S$. Then from property (a) and from (\ref{eq:21})
\begin{equation*}
\overline{P_p(e^{\frac{j\pi i}{p}}\zeta,x)}= \frac{1-|x|^{2p}}{(e^{\frac{-2j\pi i}{p}}x^2\zeta^2-2e^{\frac{-j\pi i}{p}}x\cdot\zeta+1)^{n/2}}
=\frac{1-|x|^{2p}}{|e^{\frac{-j\pi i}{p}}x-\zeta|^n}. 
\end{equation*}

Since $Z^p_m(\cdot,\zeta)$ is polyharmonic of order $p$, property (c) follows from Theorem \ref{Wn:7}.

To show (d), observe that by \cite[Proposition 1.20]{A-B-R}
\begin{gather}
\label{eq:analytic}
u(x)=\int_S\frac{1-|x|^{2}}{|x-\zeta|^n}d\sigma(\zeta)=1
\end{gather}
for $x\in B$.
Moreover, by \cite[Lemma 1]{G-M}, $u(x)$ is analytically continued to $\widehat{B}_p$ and its analytic continuation is given by the same formula.
Hence (\ref{eq:analytic}) holds for all $x\in \widehat{B}_p$. It means that
\begin{equation*}
\int_S\frac{1-|x|^{2p}}{|x-\zeta|^n}d\sigma(\zeta)=\frac{1-|x|^{2p}}{1-|x|^{2}}\int_S\frac{1-|x|^{2}}{|x-\zeta|^n}d\sigma(\zeta)=1+|x|^2+\cdots + |x|^{2(p-1)}.
\end{equation*}
Hence for $k=0,1,\dots,p-1$ we have
\begin{eqnarray*}
\int_S\frac{1-|x|^{2p}}{|e^{\frac{-k\pi i}{p}}x-\zeta|^n}d\sigma(\zeta) =
1+e^{\frac{-2k\pi i}{p}}|x|^2+\cdots + e^{\frac{-2k(p-1)\pi i}{p}}|x|^{2(p-1)}.
\end{eqnarray*}
It finishes the proof of (d).

To prove (e), observe that by the previous property
$$\frac{1}{p}\sum_{k=0}^{p-1}\int_S P_p(e^{\frac{-k\pi i}{p}}x,\zeta)\,d\sigma(\zeta)
=\frac{1}{p}\sum_{k=0}^{p-1}(1+e^{\frac{-2k\pi i}{p}}|x|^2+\cdots + e^{\frac{-2(p-1)\pi i}{p}}|x|^{2(p-1)}),$$
but
\begin{gather*}
\sum_{k=0}^{p-1}e^{\frac{-2kj\pi i}{p}}|x|^{2j}=0\quad\textrm{for}\quad j=1,\dots,p-1,
\end{gather*}
therefore
$$ \frac{1}{p}\sum_{k=0}^{p-1}\int_S P_p(e^{\frac{-k\pi i}{p}}x,\zeta)d\sigma(\zeta)=\frac{1}{p}\sum_{k=0}^{p-1}1=1,$$
as desired.

Property (f) is obvious.

To prove (g), let us note that for every $\eta \in \widehat{S}_p$ and every $\delta>0$,
$$\sum_{k=0}^{p-1}\int\limits_{\|\zeta-e^{\frac{-k\pi i}{p}}\eta\|>\delta} |P_p(e^{\frac{-k\pi i}{p}}x,\zeta) |_{\CC}\, d\sigma(\zeta)\leq 
\sum_{k=0}^{p-1}\int\limits_{\|\zeta-e^{\frac{-k\pi i}{p}}\eta\|>\delta} \frac{ 1-|x|^{2p}}{\delta^n} d\sigma(\zeta) \rightarrow 0$$
as $x\rightarrow \eta$, $x\in\widehat{B}_p$.
\end{proof}

Our next theorem states that the Dirichlet problem for $\widehat{B}_p$ is solvable for any continuous function on $\widehat{S}_p$.
\begin{Tw}
\label{Wn:9} 
Let $f\in C(\widehat{S}_p)$. If
\begin{equation}
\label{eq:23}
  u(x)= \begin{cases}
   P_p[f](x) & \textrm{for}\quad x\in \widehat{B}_p\\
   f(x) & \textrm{for}\quad x\in \widehat{S}_p,
   \end{cases}
   \end{equation} 
then $u$ is continuous on $\widehat{B}_p\cup\widehat{S}_p$ and polyharmonic on $\widehat{B}_p$.
\end{Tw}
\begin{proof}
Differentiating under the integral sign in (\ref{eq:poisson}) we obtain by (c) of Proposition \ref{Tw:5} that $P_p$ is polyharmonic in
$\widehat{B}_p$. 

So we have to show only that $u$ given by (\ref{eq:23}) is continuous. To this end take any $\varepsilon>0$ and suppose that
$\eta \in e^{\frac{j\pi i}{p}}S$ for some $j=0,1,\dots,p-1$ is a fixed point and $x\in \widehat{B}_p$. It is sufficient to find
such $\delta>0$ that $|u(x)-u(\eta)|_{\CC}<\varepsilon$ for $\|x-\eta\|_{\CC}<\delta$.

First, we assume that $x\in e^{\frac{j\pi i}{p}}B$.
Since $f$ is continuous on $\widehat{S}_p$,
we can choose $\delta_1>0$ such that $|f(e^{\frac{j\pi i}{p}}\zeta)-f(\eta)|_{\CC}< \frac{\varepsilon}{2}$ if
$||e^{\frac{j\pi i}{p}}\zeta-\eta||_{\CC^n}<\delta_1$. Moreover, the set
$\{\zeta \in S\colon ||\zeta-e^{\frac{-k\pi i}{p}}\eta||_{\CC^n}<\delta_2 \}$ is empty for $k\neq j$
if $\delta_2\leq\inf_{\zeta \in S} ||\zeta-e^{\frac{\pi i}{p}}(1,0,\dots,0)||_{\CC^n}=\sqrt{2-2\cos \frac{\pi}{p}}$.
Hence if we put $\delta_2:=\min \left\{ \delta_1,\sqrt{2-2\cos \frac{\pi}{p}} \right\}$, then by property (e) of Proposition \ref{Tw:5} 
we have
\begin{multline*} 
u(x)-u(\eta) = 
\frac{1}{p} \sum_{k=0}^{p-1}\int_{||\zeta-e^{\frac{-k\pi i}{p}}\eta||_{\CC^n}\leq\delta_2}  P_p(e^{\frac{-k\pi i}{p}}x,\zeta)
( f(e^{\frac{k\pi i}{p}}\zeta)-f(\eta))\, d\sigma(\zeta)\\
+ \frac{1}{p}\sum_{k=0}^{p-1}\int_{||\zeta-e^{\frac{-k\pi i}{p}}\eta||_{\CC^n}>\delta_2}
P_p(e^{\frac{-k\pi i}{p}}x,\zeta) (f(e^{\frac{k\pi i}{p}}\zeta)-f(\eta))\, d\sigma(\zeta).
\end{multline*} 
To estimate the first term, we apply properties (d) and (e) of Proposition \ref{Tw:5}
\begin{multline*}
\frac{1}{p}\left| \sum_{k=0}^{p-1}\int_{||\zeta-e^{\frac{-k\pi i}{p}}\eta||_{\CC^n}\leq\delta_2}  P_p(e^{\frac{-k\pi i}{p}}x,\zeta)
( f(e^{\frac{k\pi i}{p}}\zeta)-f(\eta))\, d\sigma(\zeta)\right|_{\CC}\\
\leq
 \frac{1}{p}\int_{||\zeta-e^{\frac{-j\pi i}{p}}\eta||_{\CC^n}\leq\delta_2}  P_p(e^{\frac{-j\pi i}{p}}x,\zeta)
| f(e^{\frac{j\pi i}{p}}\zeta)-f(\eta) |_{\CC}d\,\sigma(\zeta)\\
< \frac{\varepsilon}{2p}\int_S P_p(e^{\frac{-j\pi i}{p}}x,\zeta)d\,\sigma(\zeta)\leq \frac{\varepsilon}{2}.
\end{multline*}
To estimate the second term take $M:=||f||_{\infty}=\sup\limits_{\widehat{S}_p}|f|$ and observe that
\begin{multline*}
 \frac{1}{p}\left|\sum_{k=0}^{p-1}\int_{||\zeta-e^{\frac{-k\pi i}{p}}\eta||_{\CC^n}>\delta_2}
P_p(e^{\frac{-k\pi i}{p}}x,\zeta) (f(e^{\frac{k\pi i}{p}}\zeta)-f(\eta))\, d\sigma(\zeta)\right|_{\CC}\\
\leq \frac{2M}{p}\sum_{k=0}^{p-1}\int_{||\zeta-e^{\frac{-k\pi i}{p}}\eta||>\delta_2}\left| P_p(e^{\frac{-j\pi i}{p}}x,\zeta)\right|_{\CC}\,d\sigma(\zeta).
\end{multline*}
By property (g) of the last proposition we can also choose $\delta_3>0$ such that $||x-\eta||_{\CC^n}<\delta_3$ implies
\begin{eqnarray*}
\frac{2M}{p}\sum_{k=0}^{p-1}\int_{||\zeta-e^{\frac{-k\pi i}{p}}\eta||>\delta_2}\left|P_p(e^{\frac{-j\pi i}{p}}x,\zeta)\right|_{\CC}\,d\sigma(\zeta)
<\frac{\varepsilon}{2}.
\end{eqnarray*}

Therefore for every $x\in \widehat{B}_p$ such that $||x-\eta||_{\CC^n}<\delta:=\min \{\delta_2,\delta_3\}$
we conclude that $x\in e^{\frac{j\pi i}{p}}B$ and $|u(x)-u(\eta)|_{\CC}< \varepsilon$,  as desired. 
\end{proof}

\begin{Uw}
\label{Uw:Dirichlet}
Observe that Theorem \ref{Wn:9} allows  us to express the solution of the Dirichlet problem (\ref{I:1}) given in \cite[Theorem 1]{G-M}
in terms of the Poisson kernel for $\widehat{B}_p$:
\begin{equation}
\label{eq:solution}
 u(x)=P_p[f](x)=\frac{1}{p}\sum_{k=0}^{p-1}\int_S \frac{1-|x|^{2p}}{|e^{\frac{-k\pi i}{p}}x-\zeta|^n}f(e^{\frac{k\pi i}{p}})d\sigma(\zeta).
\end{equation}
\end{Uw}

\section{An explicit formula for zonal polyharmonics}
 Finally we will find the representation of zonal polyharmonics in terms of the Gegenbauer polynomials. To this end let us recall
 the following definition
\begin{Df}
\label{Df:6}
The \emph{Gegenbauer polynomial} $C^{\lambda}_m(t)$ of order $\lambda>-\frac{1}{2}$ and of degree $m$ is defined as follows:
\begin{equation}
\label{eq:24}
C^{\lambda}_m(t)=\sum_{k=0}^{[m/2]}(-1)^k\frac{\Gamma(m+\lambda-k)}{k!(m-2k)!\Gamma(\lambda)}(2t)^{m-2k}
\end{equation}
\end{Df}
We have the following generating formula for the Gegenbauer polynomials (Szeg\"o \cite[Formula (4.7.23)]{Sz}, see also
Hua \cite[Formula (7.1.2)]{H} or Morimoto \cite[Formula (5.2)]{M}):
\begin{equation}
\label{eq:25}
\sum_{m=0}^{\infty}C^{\lambda}_m(t)w^m=(1-2tw+w^2)^{-\lambda},
\end{equation}
which allows us to find an explicit formula for the zonal polyharmonics.
\begin{Tw}
\label{Tw:7}
Let $\zeta\in \widehat{S}_p$. Then $P_p(x,\zeta)$ and $Z^p_m(x,\zeta)$ can be expressed by the Gegenbauer polynomials as follows:
\begin{equation}
\label{eq:gegenbauer}
P_p(x,\zeta)=(1-|x|^{2p}) \sum_{m=0}^{\infty}C^{n/2}_m(\frac{x\cdot\overline{\zeta}}{|x||\overline{\zeta}|})|x|^m|\overline{\zeta}|^m,
\end{equation}
\begin{equation}
\label{eq:26}
Z^p_m(x,\zeta)=\left[ C^{n/2}_m(\frac{x\cdot\overline{\zeta}}{|x||\overline{\zeta}|})-
C^{n/2}_{m-2p}(\frac{x\cdot\overline{\zeta}}{|x||\overline{\zeta}|})\right]|x|^m|\overline{\zeta}|^m,
\end{equation} and in consequence we get also
\begin{multline}
\label{eq:27}
Z_m^p(x,\zeta) =  \sum_{k=0}^{\left[\frac{m}{2}\right]}(-1)^{k}\frac{n(n+2)\cdots(n+2m-2p-2k-2)}{k!2^{k}(m-2k)!}\\
\times \left[ (n+2(m-p-k))\cdots (n+2(m-k-1))+ (-1)^p2(k-p+1)\cdots 2k \right]\\
\times (x\cdot\overline{\zeta})^{m-2k}|x|^{2k}|\overline{\zeta}|^{2k}.
\end{multline}
\end{Tw}
\begin{proof}
Using (\ref{eq:25}) and putting
\begin{gather*}
t:=\frac{x\cdot\overline{\zeta}}{|x||\overline{\zeta}|}\qquad\textrm{and}\qquad w:=|x||\overline{\zeta}|
\end{gather*}
we get from Theorem \ref{Wn:7} 
\begin{equation*}
P_p(x,\zeta)=(1-|x|^{2p}) \sum_{m=0}^{\infty}C^{n/2}_m(\frac{x\cdot\overline{\zeta}}{|x||\overline{\zeta}|})|x|^m|\overline{\zeta}|^m,
\end{equation*}
and (\ref{eq:gegenbauer}) is proved.

By the same theorem we conclude that $Z^p_m$ has to be equal to the terms of degree $m$ on the right of (\ref{eq:gegenbauer}), so
\begin{equation*}
Z^p_m(x,\zeta)=\left[ C^{n/2}_m(\frac{x\cdot\overline{\zeta}}{|x||\overline{\zeta}|})-C^{n/2}_{m-2p}
(\frac{x\cdot\overline{\zeta}}{|x||\overline{\zeta}|})\right]|x|^m|\overline{\zeta}|^m,
\end{equation*}  
as desired.

From (\ref{eq:24}) we have after calculations
\begin{equation*}
C^{n/2}_m(\frac{x\cdot\overline{\zeta}}{|x||\overline{\zeta}|})=
\sum_{k=0}^{\left[\frac{m}{2}\right]}(-1)^k\frac{n(n+2)\cdots(n+2m-2k-2)}{k!2^k(m-2k)!}(x\cdot\overline{\zeta})^{m-2k}(|x||\overline{\zeta}|)^{2k-m}.
\end{equation*}
Hence
\begin{eqnarray*}
C^{n/2}_{m-2p}(\frac{x\cdot\overline{\zeta}}{|x||\overline{\zeta}|})&=&
\sum_{k=0}^{\left[\frac{m-2p}{2}\right]}(-1)^k\frac{n(n+2)\cdots(n+2m-4p-2k-2)}{k!2^k(m-2p-2k)!}\\
&\times &(x\cdot\overline{\zeta})^{m-2p-2k}(|x||\overline{\zeta}|)^{2k-m+2p}.
\end{eqnarray*}

Replacing in the above formula the index $k$ by $k-p$ we obtain
\begin{multline*}
C^{n/2}_{m-2p}(\frac{x\cdot\overline{\zeta}}{|x||\overline{\zeta}|})=
\sum_{k=p}^{\left[\frac{m}{2}\right]}(-1)^{k}\frac{n(n+2)\cdots(n+2m-2p-2k-2)}{k!2^{k}(m-2k)!}\\
\times  (-1)^p 2^p(k-p+1)(k-p+2)\cdots k(x\cdot\overline{\zeta})^{m-2k}(|x|  |\zeta|)^{2k-m}.
\end{multline*}
Since the expression  $(-1)^p 2^p(k-p+1)(k-p+2)\cdots k$ vanishes for $k=0,1,\dots,p-1$, we can sum in the last formula from $k=0$.
Finally from the above considerations and from (\ref{eq:26}) we get (\ref{eq:27}).
\end{proof}

\section{The connection with the Cauchy-Hua formula}
In Section 6 we have found the generalised Poisson kernel $P_p$ for polyharmonic functions.
In particular, for $p=1$ it is the same as the classical Poisson kernel for harmonic functions. In this section we will show 
the connection with another kernel for holomorphic functions. Namely we will find the connection between our Poisson-type formula
and the Cauchy-Hua formula. For this purpose we introduce the following definition:
\begin{Df}[see {\cite[Definition 5.5]{M}}]
The \emph{Cauchy-Hua kernel} $ H(z,w) $ is defined by
\begin{equation}
\label{eq:Hua}
 H(z,w)=\frac{1}{(\overline{w}^2z^2- 2 \overline{w}\cdot z +1)^{n/2}}.
\end{equation}
\end{Df} 

\begin{Uw}
 \label{re:Hua}
 The Cauchy-Hua kernel $H(z,w)$ is holomorphic in $z$ and antiholomorphic in $w$ on the set 
 $$
 LD:=\{(z,w)\in\CC^n\times\CC^n\colon L(z)L(w)<1\}
 $$
 and satisfies $H(z,w)=\overline{H_(w,z)}$ (see Morimoto \cite[Lemma 5.6]{M}).
\end{Uw}

 We will compare the classical Poisson kernel $P(z,w)$ defined by (\ref{eq:18}) and the Poisson kernels $P_p(z,w)$ for $\widehat{B}_p$ given by
 (\ref{eq:21}) to the Cauchy-Hua kernel (\ref{eq:Hua}).
 First, observe that a comparison of (\ref{eq:18}), (\ref{eq:21}) and (\ref{eq:Hua}) gives the following connection between the kernels
 \begin{equation}
  \label{eq:compare}
  P(z,w)=(1-z^2\overline{w}^2) H(z,w)\quad\textrm{and}\quad P_p(z,w)=(1-(z^2\overline{w}^2)^p) H(z,w).
 \end{equation}
 Hence the Poisson kernels $P(z,w)$ and $P_p(z,w)$ have the same properties from Remark \ref{re:Hua} as the Cauchy-Hua kernel.
 Moreover we conclude that
 \begin{St}
 \label{pr:conv}
  The Poisson kernels $P_p(z,w)$ converge to the Cauchy-Hua kernel $H(z,w)$ almost uniformly on $LD$ as $p$ tends to infinity.
 \end{St}
 \begin{proof}
  Using (\ref{eq:compare}) we estimate
 \begin{equation}
 \label{eq:estimate}
 (1-\|z\|^{2p}\|w\|^{2p})|H(z,w)|_{\CC}\leq |P_p(z,w)|_{\CC}\leq (1+\|z\|^{2p}\|w\|^{2p})|H(z,w)|_{\CC}.
 \end{equation}
 Since $\|z\|\leq L(z)$ for every $z\in\CC^n$, we conclude that for every compact subset $E$ of $LD$ there exists $\alpha_E<1$ 
 such that $\sup_{(z,w)\in E}\|z\|\cdot\| w\|=\alpha_E$. Hence, by (\ref{eq:estimate})
 \begin{equation}
 \label{eq:conv}
 \sup_{(z,w)\in E}|P_p(z,w)-H(z,w)|_{\CC}\leq \alpha_E^{2p}\sup_{(z,w)\in E}|H(z,w)|_{\CC}
 \end{equation}
 and the right hand side of (\ref{eq:conv}) tends to zero as $p\to\infty$, which completes the proof.
 \end{proof}
 
 \begin{Uw}
 We will improve the estimation (\ref{eq:estimate}) in the special case when $(z,w)\in LD$ is replaced by $(x,\zeta)\in \widehat{B}_p\times
 \widehat{S}_p$. To this end first observe that by Remark \ref{Re:1} we conclude that $\widehat{B}_p\times\widehat{S}_p\subset LD$,
 so $H(x,\zeta)$ is well defined. Moreover, since 
 $0<1-(x^2\overline{\zeta}^2)^p <1$, we get
 $$|P_p(x,\zeta)|_{\CC}=(1-(x^2\overline{\zeta}^2)^p) |H(x,\zeta)|_{\CC}\quad\textrm{for every}\quad (x,\zeta)\in\widehat{B}_p\times\widehat{S}_p.$$  
 \end{Uw}
 
 \bigskip

The Cauchy-Hua kernel has the following reproducing property
\begin{St}[Cauchy-Hua formula, {\cite[Theorem 5.7]{M}}]
\label{pr:3}
Let $f$ be a holomorphic function on $LB$ and continuous on $LB\cup LS$. Then 
\begin{equation*}
f(z)=\int\limits_{LS}H_1(z,w)f(w)\,d\tilde{\sigma}(w)\qquad\textrm{for}\quad z \in LB,
\end{equation*}
where the integration is performed with respect to the normalised invariant measure $d\tilde{\sigma}(w)$ on $LS$:
\begin{equation*}
\int\limits_{LS}F(w)\,d\tilde{\sigma}(w):=\frac{1}{\pi}\int\limits_{0}^{\pi}\int\limits_{S}F(e^{i\varphi }\zeta)\,
d\sigma(\zeta)d\varphi. 
\end{equation*}
\end{St}

In the next theorem we use
\begin{St} [The Siciak theorem \cite{S}, see also {\cite[Theorem 1]{F-M}} and {\cite[Theorem 3.38]{M}}]
\label{pr:2}
The Lie ball  $ LB$ is the harmonic hull of the real ball $B$, that is, every harmonic function on $ B$ is holomorphically extended to the Lie
ball $ LB$.
\end{St}

By Proposition \ref{pr:conv} it is natural to expect that the Cauchy-Hua formula is a limit of solutions of Dirichlet problems
for $p$-harmonic functions as $p$ tends
to infinity. Indeed, we have
\begin{Tw} 
\label{th:4}
Let $u$ be a holomorphic function on $LB$ and continuous on $LB\cup LS$ and let $(u_p)_{p=1}^{\infty}$ be a sequence
of polyharmonic functions of increasing orders $p$ on $B$, which are continuously extended on
$\widehat{S}_p$ putting $u_p=u$.
Then
$$
\lim_{p\to\infty}u_p(z)=\int\limits_{LS}H(z,w)u(w)\,d\tilde{\sigma}(w)=u(z)\quad\textrm{for}\quad z \in LB.
$$  
\end{Tw}
\begin{proof}
First observe that, 
since $\widehat{B}_p$ is contained in the Lie ball $LB$, by Proposition \ref{pr:2} and by the Almansi theorem \cite[Proposition 1.3]{A-C-L} 
we conclude that every polyharmonic function on $B$
is polyharmonically extended to $\widehat{B}_p$ (see also \cite[Lemma 1]{G-M}). So we may assume that $u_p$ are $p$-polyharmonic functions
on $\widehat{B}_p$ for every $p\in\NN$.
Since $(u_p)_{p=1}^{\infty}$ is a sequence of solutions of problems
\begin{equation*}
   \begin{cases}
   \Delta^pu_p(x)=0, & \text{} x\in \widehat{B}_p \\
   u_p(x)=u(x), & \mbox{} x\in \widehat{S}_p,
   \end{cases}
   \end{equation*}
by (\ref{eq:solution}) we conclude that
$$u_p(x) =\frac{1}{p}\sum_{k=0}^{p-1} \int\limits_{S}
\frac{1-\left|x\right|^{2p}}{\left|e^\frac{-k\pi i}{p}x-\zeta\right|^n } 
u(e^\frac{k\pi i}{p}\zeta) d\,\sigma(\zeta)\quad \textrm{for every}\quad  p\in \mathbb{N}. $$
Passing to the limit as $ p\rightarrow \infty $ and using the definitions of the Riemann integral and the Cauchy-Hua kernel we obtain
for $x\in B$ the sequence of equalities
\begin{eqnarray*} 
\lim_{p\rightarrow \infty}u_p(x) & =  & \frac{1}{\pi} \int\limits_{0}^{\pi} \int\limits_{S}
\frac{u(e^{i\varphi}\zeta)}{(e^{-2i\varphi}x^2\zeta^2-2e^{-i\varphi} x \cdot\zeta+1)^{\frac{n}{2}}}\,d\sigma(\zeta)d\varphi\\
& = & \frac{1}{\pi} \int\limits_{0}^{\pi} \int\limits_{S}
H(x,e^{i\varphi}\zeta)u(e^{i\varphi}\zeta)\,d\sigma(\zeta)d\varphi.
\end{eqnarray*}
So by Proposition \ref{pr:3} we conclude that
\begin{eqnarray*}
 \lim_{p\rightarrow \infty}u_p(x) &=& \int\limits_{LS}H(x, w)u(w)\,d\tilde{\sigma}(w)=u(x).
\end{eqnarray*}
By the uniqueness of the analytic continuation we may replace in the above equality $x\in B$ by $z\in LB$ and
we get the assertion.
\end{proof} 

\begin{Uw}
Let us observe that Theorem \ref{th:4} gives us the construction of the sequence of polyharmonic functions $u_p$ on $LB$, which approximates 
the given holomorphic function $u$ on $LB$ and is continuous on $LS$.
\end{Uw}


\begin{thebibliography}{30}

\bibitem{A-C-L} N. Aronszajn, T.M. Creese, L.J. Lipkin, \emph{Polyharmonic {F}unctions}, Clarendon Press, Oxford, 1983.
 
\bibitem{A-B-R} S. Axler, P. Bourdon, W. Ramey, \emph{Harmonic Function Theory}, second edition, Springer-Verlag, New York, 2001.

\bibitem{B-D-W} H. Begehr, J. Du, Y. Wang, \emph{A Dirichlet problem for polyharmonic functions}, Ann. Mat. Pura Appl. 187 (2008), 435--457.

\bibitem{D-Q-W} Z. Du, T. Qian, J. Wang, \emph{$L^p$ polyharmonic Dirichlet problems in regular domains III: the unit ball},
Complex Var. Elliptic Equ. 59 (2014), 947--965.


\bibitem{F-M} K. Fujita, M. Morimoto, \emph{Holomorphic functions on the Lie ball and related topics},
in: Finite or infinite dimensional complex analysis and applications,  35--44, Kluwer Acad. Publ., Boston, 2004.

\bibitem{G-G-S} F. Gazzola, H.Ch. Grunau, G. Sweers, \emph{Polyharmonic Boundary Value Problems}, Springer-Verlag, New York, 2010.

\bibitem{G-M} H. Grzebu{\l}a, S. Michalik, \emph{A Dirichlet type problem for complex polyharmonic functions}, Acta Math. Hungar. 153 (2017),
216--229.

\bibitem{H} L.K. Hua, \emph{Harmonic Analysis of Functions of Several Complex in Classical Domains}, Translation of Mathematical Monographs 6,
AMS, Providence, Rhode Island, 1963.

\bibitem{L} G. {\L}ysik, \emph{Higher order Pizzetti's formulas}, Rend. Lincei Mat. Appl. 27 (2016), 105--115.

\bibitem{S-M} S. Michalik, \emph{Summable solutions of some partial differential equations and generalised integral means},
J. Math. Anal. Appl. 444 (2016), 1242--1259.


\bibitem{M} M. Morimoto, \emph{Analytic Functionals on the Sphere}, Translation of Mathematical Monographs 178, AMS, Providence, Rhode Island, 1998.

\bibitem{R} W. Rudin, \emph{Principles of Mathematical Analysis}, third edition, McGraw-Hill, New York, 1976.

\bibitem{S} J. Siciak, \emph{Holomorphic continuation of harmonic functions}, Ann. Polon. Math. 29 (1974), 67--73.

\bibitem{Sz} G. Szeg\"o, \emph{Orthogonal Polynomials}, Colloquium Publications 23, AMS,  Providence, Rhode Island, 2003. 

 \end{thebibliography}
\end{document}